\documentclass[12pt]{amsart}
\usepackage{graphicx} 
\usepackage{todonotes,comment,hyperref}

\usepackage[dvipsnames]{xcolor}

\usepackage{soul}

\title{Non-expansive matrix based number systems}
\author{Adam Blažek}
\address{FNSPE, Czech Technical University, Prague, Czech Republic}
\email{\href{mailto:blazead5@cvut.cz}{blazead5@cvut.cz}}
\author{Kevin G. Hare}
\address{University of Waterloo, Department of Pure Mathematics, Canada.}
\email{\href{mailto:kghare@uwaterloo.ca}{kghare@uwaterloo.ca}}
\thanks{Research of K. G. Hare is supported, in part, by NSERC Grant 2025-03965.}
\author{Edita Pelantová}
\address{FNSPE, Czech Technical University, Prague, Czech Republic}
\email{\href{edita.pelantova@fjfi.cvut.cz}{edita.pelantova@fjfi.cvut.cz}}

\begin{document}

\def\N{\mathbb N}
\def\Z{\mathbb Z}
\def\R{\mathbb R}
\def\C{\mathbb C}

\def\A{\mathcal A}
\def\B{\mathcal B}
\def\CC{\mathcal C}
\def\D{\mathcal D}
\def\P{\mathcal P}
\def\coloneqq{:=}

\def\ii{\imath}

\def\G{\widetilde{G}}
\def\GG{\widetilde{\Gamma}}




\newtheorem{thm}{Theorem}[section]
\newtheorem{theorem}[thm]{Theorem}
\newtheorem{coro}[thm]{Corollary}
\newtheorem{corollary}[thm]{Corollary}
\newtheorem{lemma}[thm]{Lemma}
\newtheorem{cnj}[thm]{Conjecture}
\newtheorem{lmm}[thm]{Lemma}
\newtheorem{claim}[thm]{Claim}
\newtheorem{obs}[thm]{Observation}
\newtheorem{proposition}[thm]{Proposition}
\newtheorem{prop}[thm]{Proposition}
\newtheorem{definition}[thm]{Definition}
\newtheorem{note}[thm]{Notation}

\theoremstyle{remark}
\newtheorem{remark}[thm]{Remark}
\newtheorem{algo}[thm]{Algorithm}
\newtheorem{example}[thm]{Example}

\newcommand{\vect}[2]{\left(\begin{smallmatrix} #1 \\ #2 \end{smallmatrix}\right)}
\newcommand{\vectab}{\vect{a}{b}}
\newcommand{\vectp}{\vect{0}{1}}
\newcommand{\vectm}{\vect{0}{-1}}
\newcommand{\vectz}{\vect{0}{0}}
\newcommand{\wt}{\mathrm{wt}}
\newcommand{\hp}{h^{+}}
\newcommand{\hm}{h^{-}}

\begin{abstract}
An integer $d$-dimensional vector is usually represented by $d$ strings, which express the value of each component separately.
This article is devoted to number systems that allow representing integer vectors using a single string of digits.
This system is given by an integer square matrix $M$ and a finite set  $\mathcal{D}$ of symbols, called digits.  We focus on full number systems, in which every vector  $\vec{x} \in \mathbb{Z}^n$ can be  written in the form $\vec{x}=\sum_{i=0}^{k-1} M^i d_i$  with $d_i \in \D$ and  thus  represented by the string $d_{k-1}d_{k-2}\cdots d_1d_0$. 
In contrast to the very well-described case where the base $M$ is expansive, we focus on systems in which  $M$  is similar to a Jordan block with the eigenvalue $\pm 1$.  
We follow the work of Caldwell, Hare,  and V\'avra who initiated the study of such systems. First, we address the question of choosing a minimal sized digit set $\mathcal{D}$.  Then we describe optimal representations for vectors from $\mathbb{Z}^2$, both with respect to their length and the occurrence of non-zero digits. 
\end{abstract}



\maketitle


\section{Introduction}
Let $M$ be an $n \times n$ matrix and $\D$ a finite collection of vectors in $\mathbb{Z}^n$.  The  pair $(M, \D)$ is called {\em number system}.  If every vector  $\vec{x} \in \mathbb{Z}^n$ can be  written in the form 
 \begin{equation}\vec{x}=\sum_{i=0}^{k-1} M^i d_i \label{eq:rep} \text{\ \ with $d_i \in \D$},\end{equation}
 we say that  $(M, \D)$ is a {\em full number system}. The string $d_{k-1}d_{k-2}\cdots d_0$ of digits  is a representation of $\vec{x}$ in the number system.  
    If every element in $\mathbb{Z}^n$ has a unique representation, then $(M, \D)$ is called a {\em canonical number system}.
    
    The beginning of the study of such representations can be traced back to the works of Vince \cite{Vince}.   Vince noted  that $(M, \D)$ can be canonical only if the matrix   $M$ is  expansive, i.e., every eigenvalue of $M$ is in modulus $>1$.  Therefore, he focused on systems with expanding matrices.

A description of the full system in the case when the uniqueness of the representation is not required was provided by Jankauskas and Thuswaldner in \cite{JaTh}. They proved that a matrix $M$ 
can be equipped with a digit set $\D$ such that the pair $(M, \D)$ is  {full number system}    if and only if $M$  has no eigenvalue of absolute value smaller than one. 
The question of when  an expanding matrix $M$  can be equipped by a digit set $\D$  such that vector addition can be performed by a parallel algorithm is solved in  \cite{FaPeSv}. 

The study of the case where $M$ is a non-expansive, non-contractive matrix was initiated in \cite{CHV24}.  
For a matrix  $M$ similar to a Jordan block with $1$ on the diagonal, the authors solved, among other things, the question of minimal digit set  guaranteeing the fullness  of $(M, \D)$.  

Here, we  extend their study to Jordan matrices with $\pm 1$ on the diagonal and also focus on the optimal representation in terms of their length.    
We define $J_n(a)$ as the $n \times n$ Jordan matrix with eigenvalue $a$.
That is, \[J_n(a) \coloneqq \begin{pmatrix}a&1&0&\cdots&0\\0&a&1&\cdots&0\\\vdots&\ddots&\ddots&\ddots&\vdots\\0&\cdots&0&a&1\\0&\cdots&0&0&a \end{pmatrix}.\]
In this paper, we consider only $a = \pm 1$. 

In Section \ref{sec:J1} we consider the case of $J_2(1)$ with the digit set $\{\vectp, \vectm\}$.
This section answers a question raised in \cite{CHV24}. Given the matrix $J_2(1)$ and $\D = \{\vectp, \vectm\}$, what is the minimal length representation of $\vectab$.
Additionally, we give a complete description of the number of representations of $\vectab$ of length $k$ in terms of the coefficients of a simple generating function.

In Section \ref{sec:Jnm1} we show that $J_n(-1)$ along with the vectors $(0, \dots, 0, 1)^T$ and $(0, \dots, 0, 0)^T$ is a full number system.

Section \ref{sec:Jm1} studies $J_2(-1)$ with the digit set $\{\vectp, \vectz\}$.
As a consequence of Section \ref{sec:Jnm1}, this is a full number system.
Analogously to Section \ref{sec:J1}, we find the minimal length representation of $\vectab$, and show how one would construct such a representation.
We define the weight of a representation as the number of non-zero digits.
We show that the minimal weight of $\vectab$ is one of $|b|, |b|+2$ or $|b|+4$.
Lastly, we give a complete description of the number of representations of $\vectab$ in terms of the coefficients of a simple generating function.

In Section \ref{sec:lenghts} we make some observations on the relationship between the length of a representation and the norm of the vector that is represented. 

In the final section, Section \ref{sec:conc}, we make a few concluding remarks and observations.

\section{The Jordan block $J_2(1)$}
\label{sec:J1}
In this section, we consider the case of $J_2(1)$ with digit set $\{\vectp, \vectm\}$.
We let $p = \vectp$ and $m = \vectm$ for ease of notation.
In \cite{CHV24} it was noted that this is a full number system, but not a canonical number system.
In particular, all $\vectab$ have multiple representations.
The question was raised in \cite{CHV24} as to what the minimal length representation of $\vectab$ would be.
This is answered in the first subsection.

In the second subsection, we answer the related question of the number of representations of $\vectab$ of length $k$.

If $d_{k-1} \dots d_0$ is a representation of $\vectab$ then we see 
$\tilde d_{k-1} \dots \tilde d_{0}$ is a representation of $\vect{-a}{-b}$ where $\tilde p = m$ and $\tilde m = p$.
As such, we will assume without loss of generality that $b \geq 0$.

As a first useful lemma, we have:
\begin{lemma} \label{lem:ab1}
    Let $d_{k-1} \dots d_0$ be a representation in the number system $(J_2(1), \{p, m\})$ for $\vectab$. Then
    \begin{align*}
        a & = \sum_{i, d_i = p } i - \sum_{i, d_i = m} i \\
        b & = \# \{i : d_i = p \} - \# \{i : d_i =  m \}. 
    \end{align*}
\end{lemma}

\subsection{The minimal length representation}

 
\begin{remark}  Let us mention some simple properties of representations of vectors $\vectab \in \mathbb{Z}^2$ with  $b\geq 0$. 
\label{rmk:prop}
\begin{enumerate} 
\item If the letter $m$ occurs  $\ell$ times in a representation of   $\vectab$ then  the letter $p$ occurs in the representation  $b+\ell$ times. The length of the representation is $b+2\ell$.  
\label{it:2}
\item The string $p^{b+\ell}m^\ell$ represents the vector $\vectab$ with $a =\tfrac{1}{2}b(b-1) + 2b \ell +\ell^2$. 
\label{it:3}
\item The string $m^\ell p^{b+\ell}$ represents the vector $\vectab$ with $a =\tfrac{1}{2}b(b-1) -\ell^2$. 
\label{it:4}
\item Let a representation of $\vectab$ contain the factor $pm$. If we replace this factor by $mp$, we get a representation of  $\vect{a-2}{b}$. In particular, the parity of the first component of both vectors coincides. 
\label{it:5}
\end{enumerate}
\end{remark} 

\begin{note} Let $b \geq 0$ be fixed. For every $\ell \in \mathbb{N}$ we put $$\mathcal{S}_{b,\ell} = \{a \in \mathbb{Z}: \text{the letter $m$ occurs $\ell$ times in a representation of  } \vectab\}.$$
\end{note}

\begin{lemma} Let $b \geq 0$. Denote $ A_{b,\ell} = \max \mathcal{S}_{b,\ell} $  and $ a_{b,\ell} = \min \mathcal{S}_{b,\ell} $. Then \begin{itemize}
    \item  $ A_{b,\ell} =  \tfrac{1}{2}b(b-1) + 2b \ell +\ell^2$  \ and \ $a_{b,\ell} =  \tfrac{1}{2}b(b-1) -\ell^2. $
    \item The strings $p^{b+\ell}m^\ell$  and $m^\ell p^{b+\ell}$ are the unique optimal representation of the vectors   $\vect{A_{b,\ell}}{b}$  and $\vect{a_{b,\ell}}{b}$, respectively. 
\end{itemize}
\end{lemma}
    
\begin{proof} By Remark \ref{rmk:prop}, parts \eqref{it:3}, \eqref{it:4} and \eqref{it:5} and Lemma \ref{lem:ab1}.     
\end{proof}

\begin{remark}\label{rem:parity} 
Fix $b \geq 0$.
The  sequence $\{A_{b,\ell}\}_\ell$ is strictly increasing and the parity in the sequence alternates. 
Similarly,  $\{a_{b,\ell}\}_\ell$ is strictly decreasing and the parity in the sequence alternates.  
\end{remark}

\begin{lemma}\label{lem:atLeast}  Let $b \geq 0$.  Let  $a - \tfrac{1}{2}b(b-1)\equiv \ell \mod 2$ and $\ell\geq 2$.  If  
$A_{b,\ell-2} <a $ or $a< a_{b,\ell-2}$, then  
every representation of $\vectab$ contains at least $\ell$ letters $m$.
\end{lemma}
\begin{proof} This follows from the definition of $A_{b,\ell}, a_{b,\ell}$ and Remark \ref{rem:parity}. 
\end{proof}

\begin{lemma}\label{lem:atMost}   Let $b \geq 0$.    Let $\ell\geq 2$ and   $a - \tfrac{1}{2}b(b-1)\equiv \ell \mod 2$. If  
$A_{b,\ell-2} <a \leq A_{b,\ell}$ or $a_{b,\ell}\leq a< a_{b,\ell-2}$, then $\vectab$ has a representation containing $\ell$ letters $m$.
\end{lemma}
\begin{proof} First, we discuss the case $A_{b,\ell-2} < a\leq A_{b,\ell}$. 

Since $A_{b,\ell} -  \tfrac{1}{2}b(b-1)  = 2b\ell +\ell^2$ and  $a - \tfrac{1}{2}b(b-1)\equiv \ell \equiv 2b\ell +\ell^2 \mod 2$ there exist $j \in \mathbb{N}$ such that $a= A_{b,\ell} -2j$. 
As the parity of $A_{b,\ell}$ and $A_{b,\ell-2}$ is the same, $2j \leq A_{b,\ell} - A_{b,\ell-2} - 2 = 4b+4\ell - 6$, i.e., $j \leq 2b+2\ell-3$. Applying $j$ times the replacement $pm$ with $mp$ on the representation $p^{b+\ell}m^{\ell}$ we get, by Remark \ref{rem:parity}, a representation of $\vect{A_{b,\ell} - 2j}{b} = \vectab$. 
Note that it is possible to repeat the replacement $j$ times, as $j$ is bounded by 
$2b+2\ell-3$. 

The case $a_{b,\ell}\leq a< a_{b,\ell-2}$ is similar. 
\end{proof}

\begin{example} The table illustrates for $b=2$ and $\ell=2$ how we get  by the replacements $pm$ with $mp$ in  representation  of $\vect{A_{b,\ell}}{b}$   representations of vectors $\vect{a}{b}$ with 
$A_{b\ell -2}< a\leq A_{b,\ell}$ and $a = A_{b,\ell} \mod 2$.

$$\begin{array}{c|c|c}
w \in \{p,m\}^* & [w]_M\\
\hline
ppppmm & \vect{13}{2}& =\vect{A_{2,2}}{2}\\
pppmpm & \vect{11}{2}&\\
pppmmp & \vect{9}{2}&\\
ppmpmp & \vect{7}{2}&\\
 ppmmpp & \vect{5}{2}&\\
pmpmpp & \vect{3}{2}&\\   
pmmppp & \vect{1}{2}& =\vect{A_{2,0}}{2}\\ 
\phantom{pmmp}pp&\vect{1}{2}& =\vect{A_{2,0}}{2}
\end{array}
$$
Note that the leading string $pmmp$ in the penultimate row represents the zero vector, its erasing gives a shorter representation of $\vect{A_{2,0}}{2}$.     
\end{example}

\begin{thm}\label{thm:plus1}Let $b \geq 0$. Consider  $\vectab \in \mathbb{Z}^2$. Let  $\ell\in \N$ be the minimal number that satisfies 
 $$\ell \equiv a-\tfrac{b(b-1)}{2} \mod 2 \qquad \text{and} \qquad 
    -\ell^2 \leq a-\tfrac{b(b-1)}{2}\leq \ell^2 + 2b\ell. $$
Then the shortest representation of $\vectab$ has length $b + 2\ell$.    
\end{thm}  
\begin{proof}   If $\ell \geq 2$, then the statement follows from Lemmas \ref{lem:atLeast}   and \ref{lem:atMost}.    

Let $\ell = 0$. Then $a=\tfrac{1}{2}b(b-1)$. In this situation, if $b=0$, then $a =0$ and the representation of the vector is the empty word, i.e. its length is $0$.
\footnote{Typically if the zero vector is in the digit set, one would consider the length one string with digit zero to be the shortest representation of the zero vector.  As the zero vector is not in the digit set in this case, it is notationally convenient to treat the empty word as the shortest representation of zero.}
If $b> 0$, then the string $p^b$ represents the vector $\vectab $ and the representation is optimal. 

Let $\ell = 1$.  Then 
\begin{enumerate}
\item either $-1\leq a-\frac{1}{2}b(b-1) < 0$; in this case  $a =\frac{1}{2}b(b-1) - 1$ and the string $mp^{b+1}$ is the representation of   $\vectab$. 
\item or $1\leq a-\frac{1}{2}b(b-1)\leq 1+2b$. Since $a-\frac{1}{2}b(b-1)$ is odd, there exists $j \in \mathbb{N}, j\leq b$ such that $ a=\frac{1}{2}b(b-1) +1+2b - 2j$. The string $p^{b+1-j}mp^j$ represents the vector $\vectab$.   
\end{enumerate}
\end{proof}

\subsection{The number of representations}
\label{asec:J1 generating}

\begin{theorem}
Consider the full number system $(J_2(1), \{p, m\})$ with $p = \vectp$ and $m = \vectm$.
Then the number of representations of length $k$ of the term $\vectab$ is equal to the coefficient of $x^a t^b$ of the polynomial
\[ P_k(x,t) = \prod_{i=0}^{k-1} \left(t x^i + \frac{1}{t x^i}\right).\]
\end{theorem}


\begin{proof}
Let $\omega \in \{p, m\}^k$ be a word of length $k$. 
Define $\Phi: \{p,m\}^* \to \mathbb{Z}[x,t]$ by 
$\Phi(\omega) = x^a t^b$ where $[\omega]_M = \vectab$.

We claim that \[ \sum_{\omega \in \{p,m\}^k} \Phi(\omega) = P_k(x,t).\]
We prove this by induction.
This is clearly true for $k = 1$.
Assume that it is true for $k$.
Then we have
\begin{align*}
    \sum_{\omega \in \{p,m\}^{k+1}} \Phi(\omega) 
        &= \sum_{\omega' \in \{p, m\}^k} \Phi(p \omega') +  \sum_{\omega' \in \{p, m\}^k} \Phi(m \omega') \\
        & = t x^{k} P_k(x,t) + \frac{1}{t x^{k}}  P_k(x,t) \\
        & = (t x^{k} + \frac{1}{t x^{k}}) P_k(x,t) \\
        & = P_{k+1} (x,t).
\end{align*}
The third line follows from the observation that $J_{2}(1)^{k} p = \vect{k}{1}$ and $J_2(1)^{k} m = \vect{-k}{-1}$.
\end{proof}

\section{The Jordan block $J_n(-1)$}
\label{sec:Jnm1}

The following result is an analog of \cite[Lemma 2.1]{CHV24}. 
 To formulate it, we denote by the symbol  $[w]_M$   a vector from $\mathbb{Z}^n$ represented by the string 
 $w \in \mathcal{D}^*$.  
\begin{thm} \label{thm:minus1n}
  Let $n \in \N$ and $\D \subset \mathbb{Z}^n$ with $\#\D < \infty$. 
  Let $M \coloneqq J_n(-1)$. 
Assume that there exists a $z \in \D^*$ of odd length such that $[z]_M = (0, \dots, 0)^T$.
Then \((M,\D)\) is a full number system if and only if for each \(1 \le j \le n\) there exist words \(t_j,u_j \in \D^*\) such that
  \[[t_j]_M = (\tau_{j,1},\ldots,\tau_{j,j-1},T_j,0,\ldots,0)^T,\]
  \[[u_j]_M = (\upsilon_{j,1},\ldots,\upsilon_{j,j-1},U_j,0,\ldots,0)^T\]
  with \(\gcd(T_j,U_j) = 1\).
\end{thm}
\begin{remark}
    It is worth remarking that if the set of digits contains the zero digit, then there exists a $z \in \D^*$ of length 1 satisfying the second condition of the theorem.
\end{remark}
\begin{remark}
    It is unclear if $(M, \D)$ being a full number system implies that there exists a $z \in \D^*$ of odd length with $[z]_M = (0, \dots, 0)^T$.
\end{remark}
Before we prove the previous theorem, we need the following lemma:
\begin{lemma}\label{lem:V} Let $j\in \{1,2,\ldots,n\}$ be fixed. Assume that there exists a   word $z \in \D^*$ of odd length such that $[z]_M = (0, \dots, 0)^T$.  Assume further that there exists $t,u\in \D^*$
such that
  \[[t]_M = (\tau_1,\ldots,\tau_{j-1},T,0,\ldots,0)^T,\]
  \[[u]_M = (\upsilon_{1},\ldots,\upsilon_{j-1},U,0,\ldots,0)^T\]
  with \(\gcd(T,U) = 1\).
   Then for each $V \in \mathbb{Z}$ there exists a word $v \in \D^*$ such that  \[[v]_M = (\nu_1,\ldots,\nu_{j-1},V,0,\ldots,0)^T,\]
\end{lemma}
\begin{proof} By noting that $[t]_M = [z t]_M$ we will assume without loss of generality that the lengths of $t$ and $u$ are even.

  Note that for every \(k \in \N\), we have
  \[ M^k [t]_M = (\tau'_1,\ldots,\tau'_{j-1},(-1)^k T,0,\ldots,0)\]
  for some \(\tau'_1,\ldots,\tau'_{j-1}\), and analogously for \(u\). Given any \(V \in \Z\), from Bézout’s lemma we have \(x T + y U = V\) for some \(x,y \in \Z\), so
  \[[
    \underbrace{t \ \cdots \ t}_{|x|-times}
    r
    \underbrace{u\  \cdots \  u}_{|y|-times}
    s
  ]_M = (\nu_1,\ldots,\nu_{j-1},V,0,\ldots,0)^T.\]
Here $r$ and $s$ are either $z$ or the empty word, depending upon the sign of $x$ and $y$.
    
\end{proof}
\begin{proof}[Proof of Theorem \ref{thm:minus1n}]
The one direction is obvious.  
If it is a full number system, then there exists such $t_j$ and $u_j$.

To prove the other direction, assume there exist such $t_j$ and $u_j$. 
Let us realize the simple fact:

If $b, c   \in \D^*$  such that  the length of $b$ is even,   $[b]_M= (b_1,\ldots, b_n)^T$ and  $[c]_M=(c_1,\ldots, c_{j}, 0,\ldots, 0)^T$, then
for some $b_1', \ldots, b_{j-1}' \in \mathbb{Z}$
\begin{equation}\label{eq:addition}
[cb]_M= (b_1', \ldots, b_{j-1}', b_j+c_j, b_{j+1}, \ldots, b_n )^T.\end{equation}

Now we will show how for any given vector $(a_1,a_2,\ldots, a_n)^T \in \mathbb{Z}^n$  one can find its representation.

Applying  Lemma \ref{lem:V} with  $j=n$  and  $V=a_n$ we find 
a word $v=:w^{(n)}$ such that  $[w^{(n)}]_M=(w^{(n)}_1, \ldots, w^{(n)}_{n-1}, a_n)$ for some integers $w^{(n)}_1, \ldots, w^{(n)}_{n-1}$. 
Without loss of generality the length of $w^{(n)}$ is even. 

Using the same lemma with $j=n-1$
 and $V = a_{n-1}- w^{(n)}_{n-1}$  we find a word $v$  such that $[v]_M=(\nu_1, \ldots, \nu_{n-2}, V, 0)$.  The rule \eqref{eq:addition}  says that the concatenation 
 $w^{(n-1)}:=vw^{(n)}$ represents a vector of the form
 $
w^{(n-1)} =(w^{(n-1)}_1, \ldots, w^{(n-1)}_{n-2},a_{n-1}, a_n)
$. Repeating the same procedure we finally obtain a word representing the given vector  $(a_1,a_2,\ldots, a_n)^T$.
\end{proof}

Consider $J_n(-1)$, the $n \times n$ Jordan block with eigenvalue $-1$. From now on, we focus on the digit set  $\{p, z\}$ with $p = (0, 0, \dots, 0, 1)^T$ and $z$ as the zero vector.
Before we prove the main result, we need the following lemma:
\begin{lemma} \label{lem:j pos} Let $j\in \N, 1\leq j<n$. 
Assume that there exists a  word $w \in \{p, z\}^*$  such that $[w]_M = (\omega_1, \omega_2, \dots, \omega_{j}, 1, \underbrace{0, \dots, 0}_{n-j-1})^T$.  Then there exists $w'\in \{p,z\}^*$ 
such that $[w']_M = (\omega'_1, \omega'_2, \dots, \omega'_{j-1}, 1, \underbrace{0, \dots, 0}_{n-j})^T$. 
\end{lemma}
\begin{proof}  Suppose that $w$ has an even length $2m$  (prepend $z$ if necessary).  Denote  $w = a_{2m-1} \dots a_0$.
Notice that
\begin{align*}
    [w z w]_M & = \sum_{i=0}^{2m-1} M^{2m+1} M^i a_i + \sum_{i=0}^{2m-1} M^i a_i \\
    & = (M^{2m+1} + I) \sum_{i=0}^{2m-1} M^i a_i \\
    & = (M^{2m+1} + I) (\omega_1, \dots, \omega_{j}, 1, \underbrace{0, \dots, 0}_{n-j-1})^T. \\ 
\end{align*}
As $M^{2m+1}+I$ is an upper triangular matrix, with $0$ on the main diagonal, and $2m+1$ for entries $(i+1, i)$, we deduce that  
  $$ [w z w]_M =(b_1, \ldots, b_{j-1}, 2m+1,\underbrace{0, \ldots, 0}_{n-j}).
$$
Analogously,  
$[w z z z w]_M = (M^{2m+3} + I) (\omega_1, \dots, \omega_{j}, 1, \underbrace{0, \dots, 0}_{n-j-1})^T$, hence 
  $$ [w zzz w]_M =(c_1, \ldots, c_{j-1}, 2m+3,\underbrace{0, \ldots, 0}_{n-j}).
$$
Since $2m+1$ and $2m+3$ are coprime, Lemma \ref{lem:V} applied to $t=wzw$, $u=wzzzw$ and $V=1$ implies the statement. 
\end{proof}

\begin{theorem}
    Let $M = J_n(-1)$ and $\D = \{(0, \dots, 0, 1)^T, (0, \dots, 0)^T\}$. 
    Then $(M, \D)$ is a full number system.
\end{theorem}

\begin{proof}
To show the fullness of $(M, \D)$, it is enough to find words $t_j,u_j$ with the properties described in Theorem \ref{thm:minus1n}. 

Since $[p]_M = (0,\dots,0,1)$, we set  $t_n=u_n=p$. 
   By Lemma \ref{lem:j pos}, there exists a word, denote it $t_{n-1}$, such that   $[t_{n-1}]_M = (\tau_{n-1,1}, \dots, \tau_{n-1,n-2}, 1, 0)$.  We set $u_{n-1}=t_{n-1}$. 
    Repeating in this fashion, we can construct $t_j=u_j$ such that 
    $[t_{j}]_M = (\tau_{j,1}, \dots, \tau_{j,j-1}, 1, 0)$
   for each $j = 1, 2, \dots, n$, as desired. 
   
\end{proof}


\section{The Jordan block $J_2(-1)$}
\label{sec:Jm1}

In this section, we consider the case of $M=J_2(-1)$ with digit set $\{\vectp, \vectz\}$.
We let $p = \vectp$ and $z = \vectz$ for ease of notation.
We see from Section \ref{sec:Jnm1} that this is a full number system.
Noting that $[ppzpp]_M= \vectz = [z]_M$ shows that it is not a canonical number system.
In fact, all $\vectab$ have infinitely many pairwise different  representations.
The analogous question to Section \ref{sec:J1} is, what is the minimal length representation of $\vectab$.
This is answered in the first subsection.

In the second subsection, we consider instead the minimal ``weight'' representation of $\vectab$.  Here we define the weight of a representation as the number of non-zero terms in the representation.

The last subsection considers the question of the number of representations of $\vectab$ of length $k$.

Unlike Section \ref{sec:J1}, there is no obvious symmetry for these representations that allows one to easily 
    derive the representation of $\vect{-a}{-b}$ from that of $\vectab$.
For example, the minimal length representation of $\vect{3}{1}$ is $pzzpzp$ while the minimal length representation for $\vect{-3}{-1}$ is $pppzzpzpz$.

As a first useful Lemma, we have
\begin{lemma} \label{lem:ab2}
    Let $d_{k-1} \dots d_0$ be a representation in the number system $(J_2(-1), \{p, z\})$ for $\vectab$. Then
    \begin{align*}        
        b & = \#\{i: d_i = p, i\text{ even}\} - \#\{i: d_i = p, i\text{ odd}\} \\
        a & = \sum_{d_i = p, i\text{ odd}}i  - \sum_{d_i = p, i\text{ even}} i.
    \end{align*}
\end{lemma}

\subsection{The minimal length representation} 

In this section,  we denote $h\vectab$ the length of the shortest representation of the vector $\vectab$ in the number system $(J_2(-1), \{p,z\})$.  Note, that the length $0$ has only the vector $\vect{0}{0}$, which is represented by the empty string.

\begin{lemma}\label{lem:increasedecrease} Let $a,b \in \mathbb{Z}$.
\begin{enumerate} \item \label{it:4.2 1} If $h\vectab < h\vect{a+1}{b}$, 
then $\vectab$ has the unique shortest representation (up to leading zeros) of the form  
 \begin{itemize}
    \item either $(pz)^\alpha (zp)^\beta$, with $\alpha, \beta \in \N$ satisfying 
    $b =\beta-\alpha$ and $a = \alpha^2+2\alpha\beta - \beta^2 + \beta$, 
    \item or $(pz)^\alpha pp (zp)^\beta$ with $\alpha, \beta \in \N$ satisfying $b =\beta-\alpha$ and $a = \alpha^2+ 2\alpha\beta - \beta^2 + \beta + 2\alpha +1$.
\end{itemize}
 \item \label{it:4.2 2} If $h\vect{a-1}{b} > h\vectab $, then $\vectab$ has the unique shortest representation (up to leading zeros) of the form  
 \begin{itemize}
    \item either  $(zp)^\beta (pz)^\alpha$ \          with   $\alpha, \beta \in \N$,
    \item or  $(zp)^\beta zz (pz)^\alpha$ \          with   $\alpha, \beta \in \N$.
\end{itemize}
\end{enumerate}
\end{lemma}

\begin{proof}  Let $\vectab = [d_{k-1} \dots d_1 d_0]_M$, where $k=h\vectab$.
First, we notice that the string $pp z^{2j}$ represents the vector  $\vect{1}{0}$ and the string $pp z^{2j+1}$ represents the vector $\vect{-1}{0}$.

If  $h\vectab < h\vect{a+1}{b}$, then 
\begin{enumerate}
\item for all $j$, $d_{2j+1}d_{2j} \neq zz$, as  otherwise the vector  $\vect{a+1}{b}$ is represented by the string $d_{k-1}\ldots d_{2j+2} pp d_{2j-1}\cdots d_0 $  of the same length. \label{en:1}
\item for all $j$,  $d_{2j+2}d_{2j+1} \neq pp$, as otherwise  the vector  $\vect{a+1}{b}$ is represented by the string $d_{k-1}\ldots d_{2j+3} zz d_{2j}\cdots d_0 $. \label{en:2}
\end{enumerate}

 The properties  force that every shortest representation $d_{k-1}
 \dots d_0$ has the form $(pz)^\alpha (zp)^\beta$ or $(pz)^\alpha pp (zp)^\beta$. A straightforward computation gives that if $ \vectab = [(pz)^\alpha (zp)^\beta]_M$ then  $b =\beta - \alpha$  and $a = \alpha^2+ 2\alpha\beta - \beta^2 + \beta$. Analogously, if $ \vectab = [(pz)^\alpha pp (zp)^\beta]_M$ then  $b =\beta - \alpha$  and $a = \alpha^2+ 2\alpha\beta - \beta^2 + \beta + 2\alpha +1$.  
 
Note that for arbitrary $a,b$ there exists at most one pair $\alpha, \beta \in \N $ that satisfies $b= \beta-\alpha$ and $a \in \{\alpha^2+ 2\alpha\beta - \beta^2 + \beta,  \alpha^2+ 2\alpha\beta - \beta^2 + \beta + 2\alpha +1\}$. Hence, the shortest representation  of $\vectab$ is unique. 

The proof of Item \eqref{it:4.2 2} is analogous.\end{proof}

\begin{definition} Let $b \in \mathbb{Z}$ be fixed.  The sequence $(A_{b,n})_{n \in \N}$ is defined as follows: for every $n \in \N $  we put 
$$
A_{b,n} = \tfrac12 n (n+1) +\left\{ \begin{array}{ll} b -b^2, & \text{if } b\geq 0;\\ 
&\\
b^2- 2bn, & \text{if } b<0.\end{array} \right. 
$$
\end{definition}

\begin{remark}\label{rem:AlphaBeta}  Note that the sequence  $(A_{b,n})_{n \in \N}$  is  chosen to satisfy for every $\alpha, \beta \in \mathbb{N}$  with $b:= \beta -\alpha$ the following equalities

$$
\alpha^2 +2\alpha \beta - \beta^2 + \beta = \left\{ \begin{array}{ll}A_{b, 2\alpha}& \text{ if } \ b \geq 0,\\ A_{b, 2\beta}& \text{ if } \ b < 0,\end{array}\right.$$
and 
$$
\alpha^2 +2\alpha \beta - \beta^2 + \beta+2\alpha+1 = \left\{ \begin{array}{ll}A_{b, 2\alpha+1}& \text{ if } \ b \geq 0,\\ A_{b, 2\beta+1}& \text{ if } \ b  < 0.\end{array}\right.
$$
\end{remark}

\begin{coro}  Let $a,b \in \mathbb{Z}$.

\begin{enumerate}
    \item \label{it:4.5 1} If $h\vect{a-1}b > h\vectab <h\vect{a+1}{b}$, then 
    $$a=A_{b,0} \quad \text{and} \quad h\vectab = \left\{ \begin{array}{cl}2b-1& \text{ if } \ b > 0,\\ 2|b|& \text{ if } \ b  \leq  0.\end{array}\right.$$

    \item \label{it:4.5 2}
 If $h\vect{a-1}b \leq h\vectab <h\vect{a+1}{b}$, then  for some 
 $n \geq 1 $
 $$a=A_{b,n}\quad \text{and}\quad  h\vectab = 2(n+|b|).$$    

 \item  \label{it:4.5 3}
 If $h\vect{a-1}b > h\vectab \geq h\vect{a+1}{b}$,
 then  
 $$a=-b-A_{-b,n}\quad \text{and}\quad  h\vectab = 2(n+|b|)+1$$  
for some 
 $n \geq 1 $ or for $n=0$ and $b>0$. 

\end{enumerate} 
\end{coro}

\begin{proof} We apply   Lemma \ref{lem:increasedecrease}. 
\begin{enumerate}
    \item By both items of the lemma, the unique representation of $\vectab$ equals  either  $(pz)^\alpha$ or $(zp)^\beta$.
In the former case, $b=-\alpha\leq 0$ and $a=\alpha^2 = b^2 =A_{b,0}$.   
In the latter case, $b = \beta\geq 0$ and $a = -\beta^2 +\beta = -b^2 +b =A_{b,0}$.

\item Lemma \ref{lem:increasedecrease} and the previous item imply that there exist $\alpha, \beta \in \mathbb{N}$ such that $b =\beta-\alpha$ and the unique representation of $\vectab$ is either of the form $(pz)^\alpha(zp)^\beta$ with positive $\alpha, \beta$ or of the form  $(pz)^\alpha pp (zp)^\beta$ with $\alpha, \beta \geq 0$.  

By Remark \ref{rem:AlphaBeta}, in the former case,  $a= A_{b,n}$ with  $n=2\alpha$ if $b=\beta-\alpha\geq 0$  or with $n =2\beta$ if $b=\beta - \alpha<0$. Both options lead to the equality
$2(\alpha + \beta) =2(n+|b|) = h\vectab$.   

In the latter case,  $a= A_{b,n}$ with  $n=2\alpha+1$ if $b=\beta-\alpha\geq 0$  or with $n =2\beta+1$ if $b=\beta - \alpha<0$. Now, 
$h\vectab=2(\alpha + \beta+1) =2(n+|b|) $.   
\item  Lemma  \ref{lem:increasedecrease} and Item \eqref{it:4.5 1} of the corollary say that the unique representation of $\vectab$ has one of the forms: 
$(zp)^\beta(pz)^\alpha$ with  $\alpha, \beta \geq 1$ or  $(zp)^\beta zz(pz)^\alpha$ with   $\alpha\geq 0$ and $\beta \geq 1$. 
The strings can be rewritten into the form 
\[
(zp)^\beta(pz)^\alpha = z(pz)^{\beta-1}pp(zp)^{\alpha-1}z\] and  
\[(zp)^\beta zz(pz)^\alpha=z(pz)^\beta (zp)^{\alpha}z.\]

Item \eqref{it:4.2 1} of Lemma \ref{lem:increasedecrease} allows us to deduce that $\vectab = M\vect{c}{d}$ for some $\vect{c}{d} \in \mathbb{Z}^2$  satisfying $h\vect{c}{d} < h\vect{c+1}{d}$. Obviously, $c=-a-b$, $d=-b$ and $ h\vectab = h\vect{c}{d}+1$.

If $\alpha\geq 1$ in both strings, then the form of strings gives moreover $ h\vect{c-1}{d} \leq  h\vect{c}{d}$ and we apply Item \eqref{it:4.5 2} of the corollary to deduce that  $c= A_{d,n}$ for some $n \geq 1$  and thus $a= -b-A_{-b,n}$. 

If $\alpha =0$, then the string $z(pz)^{\beta}z$ represents the vector   $\vectab=\vect{-\beta^2-\beta}{\beta} = \vect{-\beta -A_{-\beta, 0}}{\beta} = M \vect{\beta^2}{-\beta}$ and thus by Item \eqref{it:4.5 1} of the corollary, 
$h\vectab = 1+ h\vect{\beta^2}{-\beta} = 2|b| +1$. 
\end{enumerate}
\end{proof}

Fix $b \in \mathbb{Z}$. By the previous corollary,   the only integers $a$ where the values $h\vectab$ change are: $A_{b,n}$ and $-b - A_{-b,n}$. 
Let us stress that the sequence $A_{b,n}$ is strictly increasing in $n$ and $-b - A_{-b,n}$ strictly decreasing. Moreover, $ -b-A_{-b,0} < A_{b,0}$ if $b> 0$ and  $ -b-A_{-b,1} < A_{b,0}$ if $b \leq  0$. 
We can conclude that $h\vectab$  is constant   for $a$ satisfying one of these conditions
\begin{enumerate}
\item   $ A_{b,n-1} < a \leq A_{b,n}$  and $n \geq 1$; 
\item   $-b - A_{-b, n}\leq a<  -b - A_{b,n-1}$ and $n \geq 2$;  
\item   $-b - A_{-b, 0}\leq a<  A_{b,0}$ and $b >0$; 
\item   $-b - A_{-b, 1}\leq a<  A_{b,0}$ and $b \leq 0$. 
\end{enumerate}
Using the explicit values of $A_{b,n}$, we can determine the length of the minimal representation.  The result is summarized in the following theorem.

\begin{theorem} \label{thm:minus1} Let $a, b \in \mathbb{Z}$.
\begin{enumerate}
\item  If $b> 0$ and  $a = -b^2+b$, then  $h\vectab = 2b-1$. 
\item If  $b\leq 0$ and  $a = b^2$, then   $h\vectab = 2|b|$. 
    \item If  $b\geq 0$ and $-b^2+b\leq a  $, then  $h\vectab = 2(b+n)$, where  $n\in \N $ is minimum  such that $a\leq -b^2+b +\frac{1}{2}n(n+1)$.
\item If  $b> 0$ and $ a < - b^2+b $, then  $h\vectab = 2(b+n)+1$, where $n\in \N $ is minimum  such that $a\geq - b^2-b-2nb- \frac{1}{2}n(n+1)$. 
\item If  $b<  0$ and $ b^2<a  $,  
    then 
    $h\vectab =2(n-b)$, where $n\in \N $ is minimum  such that $a\leq b^2 + \frac{1}{2}n(n+1) - 2nb$. 
\item If  $b\leq   0$ and $ a < b^2  $, then $h\vectab=2(n-b)+1$, where $n\in \N $ is minimum  such that $a\geq b^2 - \frac12{n(n+1)}$. 
\end{enumerate}
\end{theorem}



\subsection{The minimal weight representation} 

By $\wt\vectab$ we denote the minimal number of non-zero digits occurring in a representation of $\vectab$. We note that for all $a, b \in \mathbb{Z}$
\begin{equation}\label{eq:upper}
\wt M \vectab \leq \wt\vectab. 
\end{equation}
Noting that $M$ is invertible over the integers, we further get
\begin{equation}\label{eq:upper inv}
\wt \vectab = \wt M M^{-1} \vectab \leq \wt M^{-1}\vect ab.
\end{equation}
 
From Lemma \ref{lem:ab2} we easily deduce  that every representation $\sum_{i=0}^k M^i d_i $ of $\vectab$ has the weight at least $|b|$. 
We start our discussion  with identifying  the vectors with the smallest possible weight.

\begin{lemma}\label{lem:minimalWeight} For $a, b \in \mathbb{Z}$ we have $\wt\vectab = |b|$ if and only if  one of these three situations occurs
\begin{enumerate}
    \item $a=b=0$;
    \item  $a$ even, $a\leq - b(b-1)$ and $b>0$;
    \item  $a+b$ even, $a\geq b^2 $ and $b<0$.   
\end{enumerate}
\end{lemma}
\begin{proof}\ 
\begin{enumerate}
    \item Only the representation consisting of the digits $\vect 00$ has a weight of 0.
 i.e., $\vectab  = \vectz$.  
\item If $b>0$ and $\wt\vectab = b$,  then $$\vectab  =  \sum_{j=0}^{b-1} M^{2i_j} \,p = \sum_{j=0}^{b-1} \vect{-2i_j}{1},$$ where $0\leq i_0 < i_1<\cdots < i_{b-1}$. Hence  $a$ is even and 
  $$a=-2(i_0+i_1+\cdots+ i_{i_{b-1}})\leq -2\sum_{j=0}^{b-1} j= -b(b-1) .$$ 
  On the other hand, if  $a\leq -b(b-1) $ and $a$ is even, then  the number    $B:=a+(b-1)(b-2)$ is even and  $B\leq-2(b-1)$ thus $M^B p + \sum_{k=0}^{b-2} M^{2k} \,p$ is a representation of $\vectab$ of the weight $b$. 
\item If $b<0$ and  $\wt\vectab = |b|$, then $\vect{a'}{b'}:=M\vectab = \vect{-a+b}{-b}$ also has a representation of weight $|b|$. As $b'>0$ we apply the previous item to deduce that a representation of $\vectab$  has weight $|b|$ if and only if 
$a' =-a+b$ is even  and $ a'= -a+b \leq  -b'(b'-1) = b(-b-1)$. Or equivalently,  $a+b$ is even and $b^2\leq a$. 
\end{enumerate}
\end{proof}

\begin{theorem} The minimal weight representation of every vector $\vectab \in \mathbb{Z}^2$ satisfies 
$$
\wt\vectab \in \{|b|, |b|+2, |b|+4\}.
$$
\end{theorem}

\begin{proof}  Let $d_{k-1} \dots d_0$ be a representation of $\vectab$.
Denote 
 $$    \gamma =\#\{i:d_i=p\}\,,\quad 
     \alpha= \#\{i:d_{2i}=p\}\,\ \quad   \text{and} \quad   
  \beta=\#\{i:d_{2i+1}=p\}.   
    $$
Clearly, $\gamma =\alpha+\beta$. 
    
Noting that $M^{2i} p = 
\vect{-2i}{1}$ and $M^{2i+1} p = 
\vect{2i+1}{-1}$, we derive 
\begin{equation}\label{eq:GammaBetaAlpha}
 b = \alpha - \beta, \quad  a=\beta \!\!\mod 2 \quad  \text{and}\quad  a+b= \alpha \!\!\mod 2
\end{equation}
It is enough to discuss cases not included in Lemma \ref{lem:minimalWeight}.  
\begin{enumerate}
\item Case $b\geq 0$ and $a$ odd: \quad  We show that  $\wt\vectab=b+2$. \\
Let us choose $A\in\mathbb{N}$ large enough so that the number 
 $B = a+(b+1)(2A+b)\geq 0$. Obviously, $B$ is odd. As 
$M^B p + \sum_{i=A}^{A+b}M^{2i}p  = \sum_{i=A}^{A+b}\vect{-2i}{-1} + \vect{B}{-1} = \vect{B-(b+1)(2A+b)}{b}  = \vectab$,
we have found a representation of  $\vectab$ which has $b+2$ non-zero digits. 

On the other hand,  by Equation \eqref{eq:GammaBetaAlpha},  $\beta=1 \mod 2$  for every representation of   $\vectab$. Hence  $\beta\geq 1$ and $\gamma = \alpha+\beta = b+ 2\beta \geq b+2$. Hence, if $b\geq 0$ and $a$ is odd, then the weight of no representation can be less than $b+2$. 
\label{part:1}
  \item Case $b< 0$ and $a+b$ odd:  \quad  We show that  $\wt\vectab=|b|+2$. \\
Consider  
\[ \vect{a'}{b'} = M\vectab = \vect{-a+b}{-b} \ \ \text{and}  \ \ \vect{a''}{b''} = M^{-1}\vectab = \vect{-a-b}{-b}.\] 
By Equations \eqref{eq:upper} and \eqref{eq:upper inv},  
$$\wt\vect{a'}{b'} =  \wt\, M\vectab  \leq  \wt\vectab \leq \wt\,M^{-1}\vectab = \wt\vect{a''}{b''}.$$
Since $b'=b'' = |b|>0$ and both $a'$ and $a''$ are odd, Case \eqref{part:1} implies  $\wt\vect{a'}{b'}=\wt\vect{a''}{b''}=|b|+2$. This forces $ \wt\vectab =|b|+2$.  
\item Case $b\geq 0$, $a$ even, $a<-b(b-1)$:\quad We show  $\wt\vectab=b+4$. \\
By Equation \eqref{eq:GammaBetaAlpha} and Lemma \ref{lem:minimalWeight},  $\beta = a \!\!\mod 2$  and  $b+1 \leq \gamma = b+2\beta$. This implies  $\beta \geq 2$ and therefore $\gamma \geq b +4$.    
Clearly, 
$$ \vectab  = \vect{a-1}{b} + M^{2B+1}p + M^{2B}p
$$
for any $B \in \mathbb{N}$. By Case \eqref{part:1}, $\vect{a-1}{b}$ has a representation of weight $b+2$. Consequently,  $\vectab$  has a representation of weight $b+4$. 
\label{case:3}
\item Case $b<0$, $a+b$ even, $a>b^2$:\quad We show  $\wt\vectab=|b|+4$. \\[1mm]
By  Equation \eqref{eq:GammaBetaAlpha} and Lemma \ref{lem:minimalWeight},  $\alpha = a+b \!\!\mod 2$  and  $|b|+1 \leq \gamma = \alpha+\beta = 2\alpha +|b|$. This implies  $\alpha \geq 2$ and therefor $\gamma \geq b +4$. 

Note that the  coordinates of  $\vect{a'}{b'} :=M^{-1}\vectab = \vect{-a-b}{-b}$  satisfy $a' =-a-b \geq -b'(b'-1) =-b(b+1) $ and $a'$ even.  Case \eqref{case:3} guaranties that the vector  $\vect{a'}{b'}$ has a representation of weight $|b'|+4=|b|+4$.  By \eqref{eq:upper}, the same is true for $\vectab$.
\end{enumerate}
\end{proof}


\subsection{The number of representations}
\label{sec:J-1 generating}

Similarly to Section \ref{asec:J1 generating} we can get the number of representations of $\vectab$ of length $k$ to be the coefficient of $x^a t^b$ of
\[ \prod_{i=0}^{k-1}\left( \left(\frac{x^i}{t}\right)^{(-1)^i} + 1\right).\]

\section{Lengths of representations}
\label{sec:lenghts}

In Section \ref{sec:J1} and \ref{sec:Jm1} we provided an algorithm to find the minimal length representation with base $J_2(1)$ and $J_2(-1)$ respectively.
In this section we provide a general observation about the length of such representations.
Let $\| (a_1, \dots, a_n)^T \| = \max |a_i|$ be the $\ell_\infty$ norm.

\begin{prop}\label{prop:BoundedEntries}
    Let $M$ be a square matrix similar to $J_n(1)$ or $J_n(-1)$ and $\mathcal{D} \in \mathbb{Z}^n$ be a finite set of vectors.
    Then there exists a constant $c$ such that for all $k$ and all $w \in \mathcal{D}^k$ we have $\| [w]_M \| \leq c k^n$.
\end{prop}

\begin{proof}  Let $H = \max_{d \in \mathcal{D}} \| d \|$. 
Write  $M$ in the form  $M=P^{-1}JP$, where  $J$
 equals $J_n(1)$ or $J_n(-1)$. As entries of  $J^k$   are in modulus at most $k\choose{n-1}$, every entry of $M$ is bounded by a polynomial in variable $k$ of degree at most $n-1$. We can find a constant $F$ such that every entry of $M^k$ is in modulus bounded by  $F\,k^{n-1}$.  Hence the entries of the vector $M^id_i$  are bounded by $c\,k^{n-1}$ where $c = HF$.    Since $[w]_M= \sum_{i=0}^{k-1} M^{i}d_i$, components  of the vector $[w]_M$ are bounded by $c\,
k^n$. \end{proof}

\begin{coro}
    Let $M$ be a square matrix similar to $J_n(1)$ or $J_n(-1)$ and $\mathcal{D} \in \mathbb{Z}^n$ be a finite set of vectors of size at least 2.
    There exists a $\gamma > 0$ such that if $[w_0 \dots w_{k-1}]_M = (a_1, \dots a_n)^T =: \vec{a}$ then $k \geq \gamma \sqrt[n]{\| \vec{a}\|}$. 
\end{coro}

\begin{proof} Let $c$ be the constant from Proposition \ref{prop:BoundedEntries}.   
We have $\| \vec{a} \| = \| [w_1 \dots w_k]_M \| \leq c k^n$.
Dividing both sides by $c$ and taking $n$th roots gives $k \geq \sqrt[n]{1/c} \sqrt[n]{\|\vec{a}\|}$.
Setting $\gamma = \sqrt[n]{1/c}$ gives the result.
\end{proof}

\begin{remark}  
For $n=2$, Theorems \ref{thm:plus1} and \ref{thm:minus1} confirm that the asymptotic lower bound  $O(\sqrt{n})$ is met.
\end{remark}

\section{Final remarks} \label{sec:conc}

Consider $J_2(1)$ with digit set $\{\vectp, \vectm\}$ from Section \ref{sec:J1}.
Similarly to Section \ref{sec:Jm1}, we could define $\wt\vectab$ as the minimal number of $\vectp$ in a representation of $\vectab$.
As $b = \#\{i: d_i = p\} - \#\{i: d_i = m\}$, we see that the representation which achieves $\wt\vectab$ is the same as the representation that achieves the minimal length representation.

The case of $J_2(1)$ with digit set $\{\vectp, \vectz, \vectm\}$ is less clear with respect to various weight functions, and would be interesting to explore.

It would be also interesting to expand the description of optimal representations (with respect to both length and weight) in the case of $M=  J_n(-1) $ of general dimension $n$. 

By replacing base $J_n(1)$ with  $J_n(-1)$, we have obtained a system in which two digits are sufficient to represent all vectors from $\mathbb{Z}^n$. In the article \cite{CHV24},  it is shown that for some integer matrices $M$  similar to the Jordan block $J_n(1)$, it is necessary to use up to $n$ digits to obtain a full number system. 
We have not addressed an analogous question  for matrices similar to $J_n(-1)$ here at all.

\end{document}